\documentclass[12pt]{amsart}
 \usepackage{amssymb, amsmath,amsfonts,amsthm,amscd,textcomp,times,mathtools }
  \ifx\pdfoutput\undefined
   \usepackage[dvips]{graphicx}
   \else
   \usepackage[pdftex]{graphicx}
   \usepackage[all,cmtip]{xy}
\usepackage[T1]{fontenc}
\usepackage{calligra}
\usepackage{multido}
\usepackage{frcursive}
\usepackage{mathrsfs}
\usepackage{tikz}
   \fi
  \usepackage{epstopdf}
\usepackage{cite}
 \newtheorem{theorem}{Theorem}

\newtheorem{proposition}{Proposition}
\newtheorem{lemma}{Lemma}
\newtheorem{definition}{Definition}
\newtheorem{corollary}{Corollary}
\newtheorem{remark}{Remark}
\newtheorem{example}{Example}

\numberwithin{equation}{section}
\numberwithin{theorem}{section}
\numberwithin{proposition}{section}
\numberwithin{lemma}{section}
\numberwithin{claim}{section}
\numberwithin{corollary}{section}

\newcommand{\bull}{\ensuremath{{}\bullet{}}}
 \newcommand{\gr}{\ensuremath{\mathbb{G}(N-n-1, \cpn)}}
\newcommand{\cpn}{\ensuremath{\mathbb{P}^{N}}}
\newcommand{\slnc}{\ensuremath{SL(N+1,\mathbb{C})}}
\newcommand{\dlb}{\ensuremath{\overline{\partial}}}
\newcommand{\dl}{\ensuremath{\partial}} 
\newcommand{\ra}{\ensuremath{\longrightarrow}}
\newcommand{\om}{\ensuremath{\omega}}
\newcommand{\vp}{\ensuremath{\varphi}}
\newcommand{\vps}{\ensuremath{\varphi_{\sigma}}}

\newcommand{\xhyp}{\ensuremath{X\times\mathbb{P}^{n-1}}}
\newcommand{\lambull}{\ensuremath{\lambda_{\bull}} }
\newcommand{\mubull}{\ensuremath{\mu_{\bull}} }
\newcommand{\elam}{\ensuremath{\mathbb{E}_{\lambda_{\bull}}}}
\newcommand{\emu}{\ensuremath{\mathbb{E}_{\mu_{\bull}}}}
 \DeclarePairedDelimiter\abs{\lvert}{\rvert}
\makeatletter
\let\oldabs\abs
\def\abs{\@ifstar{\oldabs}{\oldabs*}}
 
\begin{document}
  \DeclareGraphicsExtensions{.pdf,.gif,.jpg}
\title[Heights]{ Faltings Heights, Igusa Zeta Functions, and the Stability Conjectures in K\"ahler Geometry I}
\author{Sean Timothy Paul  }
\email{stpaul@wisc.edu}
\address{Mathematics Department at the University of Wisconsin, Madison}
\date{October 11 , 2019} 
 \vspace{-5mm}  
\begin{abstract} 
Let $(X,\mathbb{L})$ be a polarized manifold. Assume that  $h_F(X)$ (the Faltings height   of $X$) is within $O(d^2)$ of $h_{\Delta}(X)$ (the height of the discriminant of $X$) for all embeddings of sufficiently large degree $d$.  Under these circumstances we prove that the Mabuchi energy of $(X,\mathbb{L})$ is proper on the full space of K\"ahler metrics in the class $c_1(\mathbb{L})$ if and only if $(X,\mathbb{L})$ is asymptotically stable. 
\end{abstract}
\maketitle 
\tableofcontents  
 \newpage
\section{\ \  Statement of Main Results}
 
Let $X^n\subset\cpn$ be a smooth, linearly normal, complex projective variety of degree $d\geq 2$. Let $\om_{FS}$ denote the Fubini-Study K\"ahler metric on $\cpn$ relative to some Hermitian metric. We set ${\om:=\om_{FS}}|_X$ .   Let $\nu_{\om}$ denote the Mabuchi energy of $(X,\om)$. We recall the following \emph{comparison theorem} which gives a complete description of the Mabuchi energy restricted to the space $\mathscr{B}_N$ of Bergman metrics associated to the embedding $X^n\subset\cpn$ . \\
 \ \\
 \textbf{Theorem.} \cite{paul2012}
\emph{Let $X^n \subset \cpn$ be a smooth, linearly normal, complex projective variety of degree $d \geq 2$ . 
 Let $R_{X}$ denote the \textbf{X-resultant} and let $\Delta_{X}$ denote the \textbf{X-hyperdiscriminant}.
 Then there is a constant $C$ depending only on $d$ , $n$ and $\om$ such that }
\begin{align*} 
\begin{split}
&\abs{\ d^2(n+1)\nu_{\om}(\varphi_{\sigma})- \left({\deg(R_X)}\log  \frac{{||\sigma\cdot\Delta_{X}||}^{2}}{{||\Delta_{X}||}^{2}} - {\deg(\Delta_{X})} \log  \frac{{||\sigma\cdot R_{X}||}^{2}}{||R_{X}||^2}\right)\ }\leq C  \quad (*) \\
\ \\
& \mbox{\emph{for all}}\ \sigma\in \slnc \ . 
 \end{split}
\end{align*}
\begin{remark}
\emph{ Our whole approach to the Stability Conjectures rests on the fact that in the comparison theorem the Mabuchi energy is scaled by $d^2$.}
\end{remark}
\begin{remark}
\emph{Bernd Sturmfels calls our $\Delta_X$ the \textbf{\emph{Hurwitz Form}} of $X\subset\cpn$ and denotes this polynomial by $Hu_X$ in \cite{sturmfels2017}. This is certainly better terminology and notation than ours.}
\end{remark}

 In the inequality $(*)$ above , $||\cdot ||$ denotes the standard $L^2$ norm on homogeneous polynomials.
 Although this was never stated explicitly, earlier work of Tian (see \cite{tian97} Lemma 8.7 pg. 32) and the author and Tian (see \cite{ags} Proposition 4.3 pg. 2576 ) provides a rather large upper bound for this constant  
 \begin{align*}
 C\lesssim 2^{\exp(d)} \ .
 \end{align*}
The best constant on the right hand side of $(*)$ can be expressed in terms of $h_{F}(X)$, the \emph{\textbf{Faltings height}} of $X$, and a new height which we denote by $h_{\Delta}(X)$ . Definitions are in the sections that follow, for now we remark that the \emph{height} is a real number that can be attached\footnote {The height requires a Hermitian metric for its definition. Number theorists also require that $V$ be defined over a number field, we do not assume this. } to any reasonably smooth, linearly normal complex subvariety $V$ of $\cpn$ .
\begin{proposition}  The optimal constant $C$ is given by
\begin{align}\label{optimal}
C=\sup_{\sigma\in G}\abs{\deg(\Delta)h_{F}(\sigma\cdot X)-\deg(R)h_{\Delta}(\sigma\cdot X)} \ .
\end{align}
\end{proposition}
A proof of this statement can be extracted from proposition 4.1 on page 277 of \cite{paul2012}. A much better conceptual explanation for it's appearance, which proves much more, can be found in the more recent article \cite{paulsergiou2019} . \\

For {any} (normal) $V\subset\cpn$ of dimension $n$ there are bounds
\begin{align*}
-\deg(R_V)\left(\sum_{j=1}^{(n+1)(N+1)}\frac{1}{j}\right)\leq h_{F}(V) \leq 0 
\ ,\ -\deg(\Delta_V)\left(\sum_{j=1}^{n(N+1)}\frac{1}{j}\right)\leq h_{\Delta}(V) \leq 0\ .
\end{align*}
 This follows at once from the well known explicit expression for the scalar Green's function of $\cpn$ for the Fubini Study metric.
The Cauchy-Schwarz inequality brings down the value of the constant $C$ in $(*)$ by several orders of magnitude 
\begin{align*}
C\lesssim d^2\log(d)+O(d^2) \ \qquad  \mbox{for} \ d\sim N.
\end{align*}
However even this bound should certainly not be the best, as it \emph{completely ignores} the scaling and the sign difference in (\ref{optimal}). To probe for a more accurate bound we remark that the limit 
\begin{align*}
\delta(\sigma(0)):=\lim_{t\ra 0}\delta(\sigma(t)) 
\end{align*}
exists for any $\sigma(t)\in SL(N+1,\mathbb{C}(t))$ , where we have defined the \emph{\textbf{height discrepancy}} by
\begin{align*}
\delta(\sigma):=\deg(\Delta)h_{{F}}(\sigma\cdot X)-\deg(R)h_{\Delta}(\sigma\cdot X) \ .
\end{align*}

Below we show that for generic $\sigma(t)$ we have the following
\begin{proposition}
\begin{align*}
|\delta(\sigma(0))|=O(d^2)  \  .
\end{align*}
\end{proposition}
This suggests that the true bound on the height discrepancy is
\begin{align*}
|\delta(\sigma)|= O(d^2) \qquad \qquad (\dagger \dagger)\ .
\end{align*}

The importance of this bound is brought out in the following Theorem.

\begin{theorem} Let $(X,\mathbb{L})$ be a polarized manifold. Let $h$ be a Hermitian metric on $\mathbb{L}$ with positive curvature $\om$. Assume that $(X,\mathbb{L})$ satisfies $(\dagger \dagger)$. Then
\begin{itemize}
\item $(X,\mathbb{L})$ is asymptotically stable if and only if the Mabuchi energy is proper on $\mathcal{H}_{\om}$. \\
\ \\
\item $(X,\mathbb{L})$ is asymptotically semistable if and only if the Mabuchi energy is bounded below on $\mathcal{H}_{\om}$. 
\end{itemize}
\end{theorem}

A variational characterization of the existence of a K\"ahler Einstein metric on a Fano manifold is provided by the following theorem of Gang Tian.
\begin{theorem}\label{properfano}(G. Tian \cite{tian97})  {Let $(X,\om) $ be a Fano manifold with $[\om]={\tt{C}}_1(X)$. Assume that $\mbox{\tt{Aut}}(X)$ is finite. Then $X$ admits a K\"ahler Einstein metric if and only if the Mabuchi energy is proper.}
 \end{theorem}
 
An important development in K\"ahler geometry is the following Theorem of Jinguri Cheng and Xiuxiong Chen, which generalizes Theorem \ref{properfano} to any K\"ahler class.\\
 
\begin{theorem}( X.X. Chen , J. Cheng \cite{chencheng1}, \cite{chencheng2}, \cite{chencheng3})   {Let $(X,\om)$ be a compact K\"ahler manifold. Then the Mabuchi energy is proper (modulo automorphisms of $X$, if any) on $\mathcal{H}_{\om}$ if and only if there is a metric of constant scalar curvature in the class $[\om]$.}
\end{theorem}

The main result of this paper is the following corollary.

\begin{corollary}Let $(X,\mathbb{L})$ be a polarized manifold satisfying $(\dagger \dagger)$ . Assume that ${\tt{Aut}}(X,\mathbb{L})$ is finite. Then $(X,\mathbb{L})$ is asymptotically stable if and only if there is a constant scalar curvature metric in $c_1(\mathbb{L})$ .
\end{corollary}

We show that $(\dagger \dagger)$ , together with the condition of asymptotic stability, allows us to \emph{take the limit} in $(*)$ of high powers of $\mathbb{L}$ and invoke Tian's Density Theorem (see \cite{tianberg}).  
The precise definition of asymptotic stability of a polarized manifold is given below.  The author's definition of  stability is quite different\footnote{As the reader will see, the definition of stability used in this article is essentially a \emph{mutatis-mutandis} extension of Mumford's definition in \cite{git} .} from the many variations of ``K-Stability'' that appear in the literature. From the author's point of view, stability is not necessarialy concerned with a variety in a projective space. Stability is a property that a pair of (non-zero) vectors in a pair of finite dimensional complex representations of an algebraic group may, or may not, possess. As we shall explain, the stability of a projective variety is a special case of this situation. Moreover, \emph{test configurations} do not play a direct part in our definition of stability, they are rather considered as a \emph{means to check} stability.  This is exactly how one parameter subgroups are used in Hilbert and Mumford's Geometric Invariant Theory. 

 \section{Semistability of Pairs}
 Let $G$ denote any of the classical linear reductive algebraic groups over $\mathbb{C}$.  Specifically $G$ can be taken to be any one of the following
\begin{align*}
SL(N+1) \ , \ SO(2N) \ ,\ SO(2N+1)\ , \ Sp(N) \ .
\end{align*}
Primarily we will be interested in the case when $G$ is the special linear group.
      For any vector space $\mathbb{V}$ and any $v\in \mathbb{V}\setminus\{0\}$ we  let $[v]\in\mathbb{P}(\mathbb{V})$ denote the line through $v$. If $\mathbb{V}$ is a $G$ module then we can consider the \emph{\textbf{projective}} orbit :     
\begin{align*}
  \mathcal{O}_{v}:=G\cdot [v]  \subset \mathbb{P}(\mathbb{V})\ .
\end{align*}
We let $ \overline{\mathcal{O}}_{v}$ denote the Zariski closure of this orbit.  

We consider pairs $(\mathbb{E}; e)$ such that $\mathbb{E}$ is a finite dimensional complex $G$-module and the linear span of the orbit $G\cdot e$ coincides with $\mathbb{E}$ .  
  The cornerstone of the author's approach to the {Stability Conjectures} is the following generalization of Mumford's Geometric Invariant Theory. The only explicit reference to the definition known to the author is \cite{smirnov2005}, the motivation seems to the problem of decomposing the {symmetric} power of an irreducible representation of $GL(n,\mathbb{C})$. It is rather mysterious that the \emph{same} definition appears\footnote{The author was led to the same definition independently. See ``stable pair'' below.} when one seeks to bound (from below) the Mabuchi energy restricted to the space of Bergman metrics. 

\begin{definition}   
$(\mathbb{U}; u)$  \textbf{ {dominates}} $(\mathbb{W}; w)$, in which case we write $(\mathbb{U}; u)\succsim (\mathbb{W}; w)$ ,  
 if and only if there exists $\pi\in Hom(\mathbb{U},\mathbb{W})^G$ such that
$ \pi(u)=w$ and the induced  rational map  
$\pi:\mathbb{P}(\mathbb{U}) \dashrightarrow  \mathbb{P}(\mathbb{W})$
restricts to a regular finite map  
$\pi:\overline{\mathcal{O}}_{u}\ra \overline{\mathcal{O}}_{w} \ $  between the Zariski closures of the orbits.
  \end{definition} 

Observe that the restriction of the map $\pi$  to $\overline{\mathcal{O}}_{u}$ is regular if and only if the following holds
  \begin{align*}
(\ast)\qquad  \overline{\mathcal{O}}_{u}\cap \mathbb{P}({\tt{ker}}(\pi))=\emptyset \ .
\end{align*}
As the reader can easily check, whenever $(\mathbb{U}; u)\succsim (\mathbb{W}; w)$ it follows that
\begin{align*} 
 & \pi(\mathbb{U})=\mathbb{W} \ \mbox{and} \ \mathbb{U}={\tt{ker}}(\pi)\oplus \mathbb{W} \ \mbox{ ($G$-module splitting) } \ .
\end{align*}
Therefore we may identify $\pi$ with projection onto $\mathbb{W}$ and $u$ decomposes as follows
\begin{align*}
v=(u_{\pi},w) \ , \ {\tt{ker}}(\pi)\ni u_{\pi}\neq 0 \ .
\end{align*}

Again the reader can easily check that $(\ast)$ is equivalent to
\begin{align*}
(\ast \ast)\qquad \overline{\mathcal{O}}_{(u_{\pi},w)}\cap\overline{\mathcal{O}}_{u_{\pi}}=\emptyset \quad \mbox{( Zariski closure in  $\mathbb{P}({\tt{ker}}(\pi)\oplus\mathbb{W}$ ) )} \ .
\end{align*}
  
We summarize this discussion in the following way. Given $\mathbb{V}$ and $\mathbb{W}$ two $G$ representations with (nonzero) points $v$ and $w$ respectively, we consider, as before, the projective orbits\footnote{ We do not assume anything about the linear spans of the orbits. }     
\begin{align*}
\mathcal{O}_{vw}:=G\cdot [(v,w)]  \subset \mathbb{P}(\mathbb{V}\oplus\mathbb{W}) \ , \ \mathcal{O}_{v}:=G\cdot [(v,0)]  \subset \mathbb{P}(\mathbb{V}\oplus\{0\})\ .
\end{align*}
  Now we can give the definition of a semistable pair. This definition seems the most appropriate for the stability conjectures\footnote{The \emph{Yau-Tian-Donaldson conjecture(s)} .} as it gives precise estimates on the Mabuchi energy restricted to the space of Bergman metrics.  
   \begin{definition}\label{semistable} 
 {The pair $(v,w)$ is \textbf{ {semistable}} if and only if}  $ \overline{\mathcal{O}}_{vw}\cap\overline{\mathcal{O}}_{v}=\emptyset $ .
  \end{definition} 
The relationship of this with Mumford's Geometric Invariant Theory is brought out in the following example.
 
 \begin{example} 
 Let $\mathbb{V}\cong \mathbb{C}$ be the trivial one dimensional representation and let $v=1$ . Suppose $\mathbb{W}$ is any representation of $G$ and let $w\in \mathbb{W}\setminus\{0\}$ . Then 
 $(1,w)$ is a semistable pair if and only if $0\notin \overline{G\cdot w}\subset \mathbb{W}$ .
 \end{example}  

 \begin{remark}
\emph{The semistability of the pair $(v,w)$ depends only on $([v],[w])$. The reader should also observe that the definition is \emph{not} symmetric in $v$ and $w$. In virtually all examples where the pair $(v,w)$ is semistable $(w, v)$ is not semistable.}
\end{remark}
\subsection{Numerical Semistability}
If the pair $(v,w)$ is semistable then obviously we have 
\begin{align}\label{numerical}
\overline{T\cdot[(v,w)]}\cap \overline{T\cdot[(v,0)]}=\emptyset 
\end{align}
for all algebraic tori $T$ of $G$ and we may as well assume (and we do) that $T$ is \emph{maximal}. In this section we relate semistability to lattice polytopes.  To begin we let $M_{\mathbb{Z}} $ be the {character lattice} of $T$
\begin{align*}
M_{\mathbb{Z}}:= \mbox{Hom}_{\mathbb{Z}}(T,\mathbb{C}^*) \ . 
\end{align*}

As usual, the dual lattice is denoted by $N_{\mathbb{Z}}$. It is well known that $ u\in N_{\mathbb{Z}}$ corresponds to an algebraic one parameter subgroup $\lambda$ of $T$. These are algebraic homomorphisms $$\lambda:\mathbb{C}^*\ra T\ .$$ The correspondence is given by 
\begin{align*}
(\cdot \ , \ \cdot) :N_{\mathbb{Z}}\times M_{\mathbb{Z}}\ra \mathbb{Z} \ , \ m(\lambda(\alpha))=\alpha^{(u , m)} \quad m\in M_{\mathbb{Z}}\ .
\end{align*}
 
 We introduce associated real vector spaces by extending scalars 
\begin{align*} 
 \begin{split}
 M_{\mathbb{R}}:= M_{\mathbb{Z}}\otimes_{\mathbb{Z}}\mathbb{R} \qquad N_{\mathbb{R}}:= N_{\mathbb{Z}}\otimes_{\mathbb{Z}}\mathbb{R}\ .
\end{split}
\end{align*}
Then the one parameter subgroups $\lambda$ of $T$ may be viewed as integral linear functionals $$l_{\lambda}:M_{\mathbb{R}}\ra \mathbb{R}\ . $$
Any rational representation $\mathbb{E}$ decomposes under the action of $T$ into  {weight spaces}
\begin{align*}
\begin{split}
\mathbb{E}=\bigoplus_{a\in {\mathscr{A} }}\mathbb{E}_{a}  \qquad \mathbb{E}_{a}:=\{e\in \mathbb{E}\ |\ t\cdot e=a(t) \ , \ t\in T\}
\end{split}
\end{align*}
$\mathscr{A} $ denotes the $T$-{support} of $\mathbb{E}$ 
\begin{align*}
\mathscr{A}:= \{a \in M_{\mathbb{Z}}\ | \ \mathbb{E}_{a}\neq 0\} \ .
\end{align*}
Given $e\in \mathbb{E}\setminus \{0\}$  the projection of $e$ into $\mathbb{E}_{a}$ is denoted by $e_{a}$. The support of any (nonzero) vector $v$ is then defined by
\begin{align*}
\mathscr{A}(e):= \{a\in \mathscr{A}\ | \ v_{a}\neq 0\} \ .
\end{align*}
\begin{definition}  \emph{ Let $T$ be any maximal torus in $G$. Let $e\in \mathbb{E}\setminus\{0\}$ . The \textbf{\emph{weight polytope}} of $e$ is the compact convex lattice polytope $\mathcal{N}(e)$ given by}
\begin{align*}\label{wtpolytope}
\mathcal{N}(e):=\mbox{ {\tt{conv}}}\  \mathscr{A}(e) \subset M_{\mathbb{R}} \ 
\end{align*}
\emph{where ``$\tt{conv}$'' denotes convex hull}.
\end{definition}
 
\begin{definition}\label{weight}
\emph{Let $\mathbb{E}$ be a rational representation of $G$. Let $\lambda$ be any 1psg of $T$  . The \textbf{\emph{weight}}  $w_{\lambda}(e)$  of $\lambda$ on $e\in \mathbb{E}\setminus\{0\}$ is the integer}
\begin{align*}
w_{\lambda}(e):= \mbox{\emph{min}}_{   x\in \mathcal{N}(e) }\ l_{\lambda}(x)= \mbox{\emph{min}} \{ (a,\lambda)\ | \ a \in \mathscr{A}(e)\}\ .
\end{align*}
 \noindent\emph{Alternatively, $w_{\lambda}(e)$ is the unique integer such that}
\begin{align*}
\lim_{|t|\rightarrow 0}t^{-w_{\lambda}(e)}\lambda(t)e \  \mbox{ {exists in $\mathbb{E}$ and is \textbf{not} zero}}.
\end{align*}
\end{definition} 
Next, given $d\in\mathbb{N}$ and $a\in\mathscr{A}$ recall that the $T$ {semi-invariants} $P\in \mathbb{C}_d{[\mathbb{E}]_a}^T$ of degree $d$
are characterized by
$$ P(\tau\cdot e)=a(\tau)P(e)\quad \mbox{for all $\tau\in T$}\ .$$

\begin{proposition} 
Let $T$ be a maximal algebraic torus of $G$, and let $\mathbb{V}$ and $\mathbb{W}$ two finite dimensional rational $G$-modules. Then the following are equivalent
\begin{align*} 
&1.\quad \overline{T\cdot[(v,w)]}\cap \overline{T\cdot[(v,0)]}=\emptyset \\
\ \\
&2.\quad \mathcal{N}(v)\subset \mathcal{N}(w) \\ 
 \ \\
&3.\quad  w_{\lambda}(v)\leq w_{\lambda}(w) \ \mbox{for all 1psg's}\ \lambda :\mathbb{C}^*\ra T \\
\ \\
  &4.\quad \mbox{For every $\chi\in\mathscr{A}(v)$}\ \mbox{there is an}\   f\in  \mathbb{C}_d[\ \mathbb{V}\oplus\mathbb{W}\ ]^T_{d\chi} \\
  & \qquad\mbox{such that}\ f((v,w))\neq 0\ \mbox{and}\ f|_{\mathbb{V}}\equiv 0
    \end{align*}
\end{proposition}
 \ \\
In order to define a \emph{\textbf{strictly stable}} (henceforth stable) pair we need a large (but fixed) integer $m$ and the auxiliary left regular representation of $G$
\begin{align*}
G\times\mathcal{GL}(N+1,\mathbb{C}) \ \ni \ (\sigma , A)\ra \sigma\cdot A \ .
\end{align*}
  Recall that $\mathcal{GL}(N+1,\mathbb{C})$ is the vector space of square matrices of size $N+1$. The action is matrix multiplication. 
The {{standard $N$-simplex}}, denoted by $Q_N$, is defined to be the weight polytope of  the identity operator 
\begin{align*}
\mathbb{I}\in \mathcal{GL}(N+1,\mathbb{C})\ .
\end{align*}
 $Q_N$  is full-dimensional and contains the origin in its strict interior
\begin{align*}
0\in Q_N:=\mathcal{N}(\mathbb{I})\subset M_{\mathbb{R}} \ .
\end{align*}  

Let $\mathbb{V}$ be a $G$ module. We define the \emph{\textbf{degree}} of $\mathbb{V}$  as follows
\begin{align*}
\deg(\mathbb{V}):=\min\Big\{k\in \mathbb{Z}_{>0} \ |\ \mathcal{N}(v)\subseteq kQ_N\ \mbox{for all $0\neq v\in \mathbb{V}$} \ \Big\} \ .
\end{align*}
For any $v\in \mathbb{V}$, $w\in \mathbb{W}$, and $m\in \mathbb{N}$ we define
\begin{align*}
& v^m:=v^{\otimes m}\in \mathbb{V}^{\otimes m} \\
\ \\
& w^{m+1}:=w^{\otimes (m+1)}\in \mathbb{W}^{\otimes (m+1)}\\
\ \\
& \mathbb{I}^q:=\mathbb{I}^{\otimes q}\in \mathcal{GL}(N+1,\mathbb{C})^{\otimes q} \ .
\end{align*}
Finally we can give the definition of a stable pair. 
\begin{definition}\label{stable} 
\emph{The pair $(v,w)$ is \emph{\textbf{stable}} if and only if there is a positive integer $m$ such that $(\mathbb{I}^q\otimes v^m \ , \ w^{m+1})$ is semistable where $q$ denotes the degree of $\mathbb{V}$. }
\end{definition}
 \subsection{Finite dimensional energies}
 Next we endow $\mathbb{V}$ and $\mathbb{W}$ with Hermitian norms.  Using these norms we introduce the finite dimensional Mabuchi and Aubin functionals 
 \begin{align*}
  &\nu_{(v,w)}(\sigma):=\log\frac{||\sigma\cdot w||^2}{||w||^2}-\log\frac{||\sigma\cdot v||^2}{||v||^2}  \\
  \ \\
  & J_{v}(\sigma):=\deg(\mathbb{V})\log\frac{||\sigma||^2}{N+1}-\log\frac{||\sigma\cdot v||^2}{||v||^2}\ .
\end{align*}
 \begin{definition}
 \emph{The Mabuchi energy of the pair $(v,w)$ is \textbf{\emph{proper}} if and only if there are constants $\varepsilon>0$ and $b$ such that}
 $\nu_{(v,w)}(\sigma)\geq \varepsilon J_{v}(\sigma) +b$  {\emph{for all $\sigma\in G$ .}} \ 
 \end{definition}
 The following proposition relates the behavior of the energy with the semistability of the pair. Despite its simplicity, it is at the heart of the author's approach to the Stability Conjectures in K\"ahler geometry.
 \begin{proposition}\label{heart} The pair $(v,w)$ is semistable if and only if the Mabuchi energy $\nu_{(v,w)}$ is bounded below and it is stable if and only if the Mabuchi energy is proper.
 \end{proposition}
 \begin{proof}
 Observe that $$\inf_{\sigma\in G}\nu_{(v,w)}(\sigma)=\log\tan^2 {\tt{dist}}(\overline{\mathcal{O}}_{vw} ,\overline{\mathcal{O}}_{v}) \ ,$$  where ${\tt{dist}}$ denotes the distance in the Fubini-Study metric associated to the norms on $\mathbb{V}$ and $\mathbb{W}$.
 \end{proof}
We end this section with a direct comparison of Mumford's stability and the author's stability of pairs. Observe that the left hand column of the table below arises from the right when we take $\mathbb{V}\cong\mathbb{C}$ (the trivial one dimensional representation) and $v=1$. 
\ \\
 \begin{center} \begin{tabular}{l|l}
\qquad  \textbf{Mumford's G. I. T. \ }&\qquad \qquad \textbf{ Pairs} \\ \\
\hline \\
 For all $T\leq G$ $\exists\ d\in \mathbb{Z}_{>0}$ and & For all $T\leq G$ and $\chi\in\mathscr{A}(v)$\\
 $f\in \mathbb{C}_{\leq d}[\ \mathbb{W}\ ]^T$ such that   & $\exists\ d\in \mathbb{Z}_{>0}$ and $f\in  \mathbb{C}_d[\ \mathbb{V}\oplus\mathbb{W}\ ]^T_{d\chi}$  \\
   $f(w)\neq 0$ and $f(0)=0$ & such that $f((v,w))\neq 0$ and $f|_{ \mathbb{V}}\equiv 0$ \\  \\
 \hline \\
 $0\notin\overline{G\cdot w}$ & \ $\overline{\mathcal{O}}_{vw}\cap\overline{\mathcal{O}}_{v}=\emptyset$  \\ \\
 \hline \\
 $w_{\lambda}(w)\leq 0$   &\ $w_{\lambda}(w)-w_{\lambda}(v)\leq 0$   \\
 for all 1psg's $\lambda$ of $G$ & \ for all 1psg's $\lambda$ of $G$\\ \\
 \hline \\
 $0\in \mathcal{N}(w)$ all $T\leq G$ &\ $\mathcal{N}(v)\subset \mathcal{N}(w)$ {all $T\leq G$} \\ \\
 \hline \\
 $\exists$ $C\geq 0$ such that &\ $\exists$ $C\geq 0$ such that \\
 $\log||\sigma\cdot w||^2\geq -C$   &\ $\log {||\sigma\cdot w||^2}-\log{||\sigma\cdot v||^2}  \geq -C $\\  
  all $\sigma\in G$ &\ all $\sigma\in G$ \\ \\
  \hline\\
  $G\cdot w$ \emph{closed} and $G_{w}$ finite & $\exists m\in\mathbb{N}$ such that $(\mathbb{I}^q\otimes v^m,w^{m+1})$ is semistable \\
 \end{tabular} 
  \end{center} 
\section{Stability of Projective Varieties}
Fix $L\subset \mathbb{C}^{N+1}$ , $\dim(L)=n+1<N+1$.  Choose  $l\in \mathbb{N}$ satisfying $0\leq l\leq n$. Consider the Zariski open subset $\mathscr{U}$ of the Grassmannian defined by
\begin{align*}
\begin{split}
 \mathscr{U}:=\{E\in G(N-l \ ,\ \mathbb{C}^{N+1})\ |\  H^{\bull}\left(0\ra E\cap L\ra E \overset{\pi_{L}}{\ra}\mathbb{C}^{N+1}/L\ra 0 \right)=0\}\ .
 \end{split}
 \end{align*}
Observe that  $E\in\mathscr{U}$ if and only if
\begin{align*}
\dim(\pi_{L}(E))=N-n \ .
\end{align*}
 
By the rank plus nullity theorem we have that for \emph{any} $E\in G(N-l \ ,\ \mathbb{C}^{N+1})$
\begin{align*}
\dim(E\cap L)+\dim(\pi_{L}(E))=N-l \ .
\end{align*}
 Therefore $E\in \mathscr{U}$ if and only if $\dim(E\cap L)\leq n-l  $ .
 Motivated by this we define a subvariety $Z(L)$ of our Grassmannian by
\begin{align*}
Z(L):=G(N-l \ ,\ \mathbb{C}^{N+1})\setminus \mathscr{U}=\{E\in G(r \ , \ \mathbb{C}^{N+1})\ |\ \dim(E\cap L)\geq n-l+1\ \} \ .
 \end{align*}

 Now we apply the previous linear algebra to a projective variety $X^n\subset\cpn$. Recall that for any $p\in X$ that the \emph{embedded tangent space} to $X$ at $p$ is the $n$ dimensional \emph{projective} linear subspace 
 \begin{align*}
 \mathbb{T}_p(X)\in \mathbb{G}(n\ ,\ \cpn) \ 
 \end{align*}
  obtained (for example) by projectivizing the tangent space the the cone over $X$ at any point $v\in\mathbb{C}^{N+1}\setminus \{0\} $ lying over $p$.

   Given any $0\leq l \leq n$ we define the following subvariety $Z_{l+1}(X)$ of the Grassmannian by
 \begin{align*}
 Z_{l+1}(X):=
 \{E\in \mathbb{G}(N-(l+1),\cpn)\ | \  \exists \ p\in X\cap E\ \mbox{\emph{and}}
 \ \dim(E\cap \mathbb{T}_p(X))\geq n-l\} \ .
 \end{align*}
 Generally $Z_{l+1}(X)$ has \emph{codimension one}  in $\mathbb{G}(N-(l+1),\cpn)$ .

 To make the {defining polynomial} of $Z_{l+1}(X)$ concrete we view the Grassmannian in primal Stiefel coordinates \cite{sturmfels2017} by observing that there is a dominant map \footnote{The superscript $o$ denotes matrices of maximal rank.}
 \begin{align*}
 M_{(l+1)\times (N+1)}^{o}\ni A\ra \pi({\tt{ker}}(A))\in \mathbb{G}(N-(l+1),\cpn)\ .
 \end{align*}
 
 We may then consider the divisor (also denoted by $Z_{l+1}(X)$ )
 \begin{align*}
 \overline{\pi^{-1}(Z_{l+1}(X))}\subset M_{(l+1)\times (N+1)} \ .
 \end{align*}

Our ``new''  $Z_{l+1}(X)$ is now an irreducible algebraic hypersurface  in the affine space $M_{(l+1)\times (N+1)}$ and hence is cut out by a single polynomial $f_{l+1}$
\[ Z_{l+1}(X)=\{ (a_{ij}) \in  M_{(l+1)\times (N+1)}\ |\ f_{l+1}(a_{ij})=0\}\ .\]
We should point out that \footnote{The height discrepancy arises naturally from this point of view.} $Z_{l+1}(X)$ is dominated by the variety of zeros of a larger system $I_\mathscr{X}$  in more variables $p\in \mathscr{X}$  where $\mathscr{X}$ is an auxiliary projective variety naturally associated to $X$.
 \[ I_\mathscr{X}=\{(p,\ (a_{ij}))\in \mathscr{X}\times M_{(l+1)\times (N+1)}\ |\ s_k(p,(a_{ij}))=0 \ ; \ 1\leq k\leq m \} \]
The situation can be visualized as follows
\begin{align*} 
\xymatrix{  
	I_\mathscr{X} \ar@{^{(}->}[r]^-{\iota} \ar[d]^{{\pi_2}|_{I_\mathscr{X}}} &   {\mathscr{X}}\times M_{(l+1)\times (N+1)}\ar[d]^{\pi_2} \\  
	Z_{l+1}(X)\ar@{^{(}->}[r] & M_{(l+1)\times (N+1)} \ .
	 }
\end{align*}

In geometric terms $(p,(a_{ij}))\in I_\mathscr{X}$ if and only if ${\tt{ker}}(a_{ij})$ fails to meet $\mathscr{X}$ generically at $p$ (and possibly at some other point $q$ ) .
 $Z_{l+1}(X)$ is therefore the \emph{resultant system} obtained by eliminating the variable $p$ from $I_{\mathscr{X}}$. Since $\mathscr{X}$ is projective, $Z_{l+1}(X)$ is a subvariety of $M_{(l+1)\times (N+1)}$ (see \cite{gkz} , \cite{weyman}, \cite{paulsergiou2019} ) .

 \subsection{Resultants}
Let $X^n\subset\cpn$ be an irreducible, $n$-dimensional, linearly normal, complex projective variety of degree $d$ .
  \begin{definition} (Cayley 1840's) \emph{The \textbf{\emph{associated hypersurface}} to $X^n\subset \cpn$ is given by}
\begin{align*}
Z_{n+1}(X)=\{L\in \mathbb{G}(N-n-1,N)\ | L\cap X\neq \emptyset \} \ .
\end{align*}
\end{definition}
As we have remarked, it is known that $Z_{n+1}(X)$ enjoys the following properties \\
\ \\
$i)$   $Z_{n+1}(X)$ is a {divisor} in $\mathbb{G}(N-n-1,N)$ ( and hence $M_{(n+1)\times (N+1)}$ ) . \\
 \ \\
$ii)$   $Z_{n+1}(X)$ is irreducible . \\
  \ \\
$iii)$  $\deg(Z_{n+1}(X))=d$ ( $=d(n+1)$ in Steifel coordinates ) . \\
  \ \\
  Therefore there exists $R_X\in H^0(\mathbb{G}(N-n-1,N), \mathcal{O}(d))$ such that
 \begin{align*}
 \{R_X=0\}=Z_{n+1}(X)
 \end{align*}
 $R_X$ is the Cayley-{Chow} form of $X$. Modulo scaling,  $R_X$ is unique . Following the terminology of Gelfan'd \cite{gkz} we call $R_X$ the \textbf{$X$-\emph{resultant}} .  From our dual Steifel point of view we will always view $R_X$ as a polynomial \footnote {It is necessarialy invariant under the natural action of $SL(n+1,\mathbb{C})$ .} in the matrix entries   
\begin{align*}
R_X\in \mathbb{C}_{d(n+1)}[M_{(n+1)\times (N+1)}]^{SL(n+1,\mathbb{C})} \ . 
\end{align*}
\subsection{Hyperdiscriminants }  
Assume that $X\subset\cpn$ has degree $d\geq 2$. Let $X^{sm}$ denote the smooth points of $X$. For $p\in X^{sm}$ let 
$\mathbb{T}_p(X)$ be the {embedded tangent space} to $X$ at $p$ .
\begin{definition}
 \emph{ The \emph{\textbf{dual variety}} of $X$, denoted by $X^{\vee}$, is the Zariski closure of the set of \emph{tangent hyperplanes} to $X$ at its smooth points }
\begin{align*}
X^{\vee}=\overline{ \{ f\in {\cpn}^{\vee} \ |  \  \mathbb{T}_p(X)\subset {\tt{ker}}(f) \ , \ p\in X^{sm}\} } \ .
\end{align*}
\end{definition}
 Generally $X^{\vee}$ is codimension one in $ {\cpn}^{\vee}$. This holds, for example, whenever $X$ is a (nonlinear) projective curve or surface. Observe that we have the identity $$ X^{\vee}=Z_1(X)\ . $$ 
 
 For the purposes of understanding the Mabuchi energy, what is important is not the dual variety $X^{\vee}$ but the variety $Z_{n}(X)$.
 Observe that like the Cayley divisor and the dual variety $Z_{n}(X)$  also has a simple geometric description
 $$Z_{n}(X)=\{ L\in \mathbb{G}(N-n \ ,\ \cpn)\ |\ \#(L\cap X)\neq \deg(X)\} $$
 
 It is known that $Z_{n}(X)$ enjoys the following properties \\
\ \\
$i)$   $Z_{n}(X)$ is a {divisor} in $\mathbb{G}(N-n,N)$ ( and hence $M_{n\times (N+1)}$ ) . \\
 \ \\
$ii)$   $Z_{n}(X)$ is irreducible . \\
  \ \\
$iii)$  $\deg(Z_{n}(X))=n(n+1)d-d\mu$ in Steifel coordinates  . \\
  \ \\
  Therefore there exists $\Delta_X\in H^0(\mathbb{G}(N-n,N), \mathcal{O}((n+1)d-d\frac{\mu}{n}))$ such that
 \begin{align*}
 \{\Delta_X=0\}=Z_{n}(X)
 \end{align*}
   Modulo scaling,  $\Delta_X$ is unique. Inspired by the terminology of Gelfan'd  we call $\Delta_X$ the \textbf{$X$-\emph{hyperdiscriminant}}.  From our primal Steifel point of view we will always view $\Delta_X$ as a polynomial \footnote {It is necessarialy invariant under the natural action of $SL(n,\mathbb{C})$ .} in the appropriate matrix entries   
\begin{align*}
\Delta_X\in \mathbb{C}_{n(n+1)d-d\mu}[M_{n\times (N+1)}]^{SL(n,\mathbb{C})} \ . 
\end{align*}
 
 We summarize these constructions in the following proposition.
 \begin{proposition} {
 Let $X^n\subset\cpn$ be a smooth, linearly normal complex projective variety. There exists dominant integral weights $\lambull$ , $\mubull$ (with corresponding irreducible $G$-modules $\elam$ , $\emu$) and $G$-equivariant associations 
 \begin{align*}
& X  \Rightarrow R_X  \in \elam \ , \ (n+1)\lambda_{\bull}= \big(\overbrace{ {\deg(R_X)} , {\deg(R_X)} ,\dots, {\deg(R_X)}}^{n+1},\overbrace{0,\dots,0}^{N-n}\big) \\
\ \\
& X \Rightarrow \Delta_X  \in \emu \ , \  n\mu_{\bull}= \big(\overbrace{ {\deg(\Delta_X)} , {\deg(\Delta_X)} ,\dots, {\deg(\Delta_X)}}^{n },\overbrace{0,\dots,0}^{N+1-n}\big)  
\ .
\end{align*}}
 \end{proposition}
 Of course in the above proposition we know that
 \begin{align*}
 &\elam\cong \mathbb{C}_{d(n+1)}[M_{(n+1)\times (N+1)}]^{SL(n+1,\mathbb{C})} \\
 \ \\
 &\emu \cong \mathbb{C}_{n(n+1)d-d\mu}[M_{n\times (N+1)}]^{SL(n,\mathbb{C})} \ .
 \end{align*}
  For our purpose we must \emph{normalize the degrees} (so to speak) of  these polynomials. From this point on we are interested in the pair 
\begin{align*}
R:= {R_{X}}^{\otimes\deg(\Delta_{X})}\ ,\ \Delta:={\Delta_{X}}^{\otimes\deg(R_X)} \ .
\end{align*}
 
 Now we are prepared to make the following definition.
 \begin{definition} \label{stabilityofvarieties} 
\emph{Let $X\subset\cpn$ be a smooth, irreducible, linearly normal complex projective variety. Then $X$ is  \emph{\textbf {semistable}}  if and only if the pair $(R\ ,\ \Delta)$  is semistable for the action of $G$. Explicitly, the orbit closures are disjoint }$$
\overline{\mathcal{O}}_{R\Delta}\cap \overline{\mathcal{O}}_{R}=\emptyset \ .$$
 \end{definition}

Now we introduce \emph{asymptotic} semistability of a polarized manifold $(X,\mathbb{L})$. We require an auxiliary Hermitian metric $h$ on $\mathbb{L}$ with positive curvature $\om_h$.  The definition of asymptotic semistability is independent, in the obvious way, of which $h$ is chosen. Below, both $R$ and $\Delta$ have been scaled to unit length.

\begin{definition} \label{asymptoticsemistability} 
 \emph{A polarized manifold $(X,\mathbb{L})$ is \emph{\textbf { asymptotically semistable}} if and only if there is a uniform constant $C=C(h)$ such that} 
 \begin{align}
 {\tt{dist}}(\overline{\mathcal{O}}_{R\Delta},\overline{\mathcal{O}}_{R})\succsim \exp(-Cd^2) \ 
 \end{align}
 \emph{for all sufficiently large $\mathbb{L}^k$-embeddings of degree} $d=k^n$ .
 \end{definition}
 It should be emphasized that the orbit closures {must} be disjoint for all powers of $\mathbb{L}$, otherwise the Mabuchi energy is unbounded from below and no canonical metric exists. Asymptotic semistability not only requires orbit closure separation for each embedding, but crucially that \emph{the orbit closures are not allowed to approach one another too quickly as the degree of the embedding increases}.
 \begin{definition}
 \emph {Let $X\subset\cpn$ be a smooth, irreducible, linearly normal complex projective variety. Then $X$ is  \emph{\textbf {stable}}  if and only if the pair $(R,\Delta)$ is stable for the action of $G$.  Explicitly, there is an integer $m$ such that the pair $$
 (\mathbb{I}^q\otimes R^{\otimes (m-1)}\ , \ \Delta^{\otimes  m})$$
 is semistable for the action of $G$ and $q=\deg(R_X)\deg(\Delta_X)$.}
 \end{definition}
 
Again we must equip $\mathbb{L}$ with a Hermitian metric as above.
 \begin{definition} \label{asymptoticstability} 
 \emph{A polarized manifold $(X,\mathbb{L})$ is \emph{\textbf { asymptotically stable}} if and only if there are uniform constants $m\in\mathbb{Z}_{>0}$ and $C=C(h,m)$ such that} 
 \begin{align}
 {\tt{dist}}(\overline{\mathcal{O}}_{(v,w)},\overline{\mathcal{O}}_ {v})\succsim \exp(-Ck^{2n+1}) \ 
 \end{align}
 \emph{for all sufficiently large $k$ (the power of the embedding) .}
 $$(v,w):= (\mathbb{I}^q\otimes R^{\otimes(km-1)}\ , \ \Delta^{\otimes km}) \ .$$
 \end{definition}
  
\begin{remark}\emph{As in the definition of asymptotic semistability, both $R$ and $\Delta$ have been scaled so that both $v$ and $w$ are of unit length. The reader should observe that the speeds of approach in the definitions and asymptotic stability and semistability differ by a single factor of $k$ . }
\end{remark}
\section{Heights of polynomials and Igusa Local Zeta Functions}
Let $P(z_0,z_1,\dots,z_N)$ be a homogeneous polynomial of degree $d$ on $\mathbb{C}^{N+1}$.   Equip  $\mathbb{C}^{N+1}$ with it's standard Hermitian metric
\begin{align*}
<Z,W>:=z_0\bar{w}_0+\dots +z_N\bar{w}_N \ .
\end{align*}
This in turn induces the Fubini-Study K\"ahler metric $\om_{FS}$ on $\cpn$ as well as a Hermitian metric $h_{FS}^d$ on all tensor powers $\mathscr{O}(d)$ of the hyperplane bundle. We may view $P$ as a section of this bundle
\begin{align*}
P\in H^0(\cpn , \mathscr{O}(d)) 
\end{align*}
and we define the $L_2$ norm of $P$ in the usual way
\begin{align*}
||P||^2:=\int_{\cpn}|P|_{h^d_{FS}}^2\om_{FS}^N \ .
\end{align*}
Recall that the pointwise norm of $P$ is given by
\begin{align*}
|P|_{h^d_{FS}}^2([z_0\dots z_N])=\frac{|P(z_0\dots z_N)|^2}{(|z_0|^2+|z_1|^2+\dots+|z_N|^2)^d} \ .
\end{align*}

\begin{definition}
The \textbf{height} of $P$ is defined to be the real number given by
\begin{align*}
h(P):=-\log ||P||^2+ \int_{\cpn}\log|P|_{h^d_{FS}}^2\om_{FS}^N
\end{align*}
\end{definition}
We remark that $h$ is a function on $\mathbb{P}(H^0(\cpn , \mathscr{O}(d)))$ .
\begin{proposition} (see lemma 8.7 from \cite{tian97})\\ 
\ \\
The function
$
\mathbb{P}(H^0(\cpn , \mathscr{O}(d)))\ni P\ra h(P) 
$
is H\"older continuous and moreover satisfies the explicit bounds
\begin{align*}
-d(\sum_{j=1}^{N-1}\frac{1}{j})\leq h(P)\leq 0 \ .
\end{align*}
\end{proposition}

The space of matrices can be equipped with the standard Hermitian inner product induced by the Fubini-Study metric on $\cpn$, therefore we may introduce two height functions of $X\subset \cpn$.
\begin{definition} (see \cite{bostgilletsoule})
The \textbf{Faltings height} of $X\subset \cpn$ is the height of the Cayley form $$h_F(X):=-\log||R_X||^2+\int_{\mathbb{P}M_{(n+1)\times (N+1)}  }\log|R_X|_{h^{\deg(R)}_{FS}}^2\om_{FS}^{(n+1)(N+1)}\ .$$
\end{definition}
Similarly, we introduce another height function of a projective variety
\begin{definition}  
 $h_\Delta(X)$ denotes the height of the hyperdiscriminant $$h_\Delta(X):=-\log||\Delta_X||^2+\int_{\mathbb{P}M_{n\times (N+1)}}\log|\Delta_X|_{h^{\deg(\Delta)}_{FS}}^2\om_{FS}^{n(N+1)}\ .$$
\end{definition}
Given two heights it is natural, and in our case absolutely necessary, to compare them.

\begin{definition} The \textbf{height discrepancy} $\delta(X)$ of $X\subset \cpn$ is the real number given by
$$\delta(X):=|\deg(\Delta)h_F(X)-\deg(R)h_{\Delta}(X)| \ .$$
\end{definition}

As the reader might imagine, if it were easy to compute the height then all of our troubles would be over. Unfortunately, direct computation is quite difficult, even for relatively simple polynomials. To partially address this issue we introduce the local zeta functions of the title. The author learned of this point of view from an article of V. Maillot and J.Cassaigne \cite{cassaigne-maillot} .

 \begin{definition} 
 Let $P\in\mathbb{C}_d[z_0,z_1,\dots , z_N]$ and $s\in \mathbb{C}$ have $\Re(s)>0$. The Igusa-Sato \textbf{local zeta function} $Z(P;s)$ attached to $P$ is
  \begin{align*}
\frac{\Gamma(N+1 +ds)}{\Gamma(N+1)}Z(P;s):=\int_{\mathbb{C}^{N+1}}e^{-||\vec{z}\ ||^2} {|P(\vec{z}\ )|^{2s}}\frac{dV}{\pi^{N+1}} \ .
\end{align*}
\end{definition}
Work of Atiyah, Gelfan'd and Bernstein, and Bernstein  ( \cite{atiyah(resolution)} ,\cite{bernsteingelfand},\cite{bernstein}) shows that $Z(P;s)$ extends to a holomorphic function of $-\varepsilon <\Re(s)$, admits a meromorphic extension to all of $\mathbb{C}$, and satisfies a functional equation. 
The relationship between heights and zeta functions is given by
\begin{align*}
 h(P)= -\log Z(P;1)+ \ Z^{\prime}(P;0)    \ .
\end{align*}

As we have remarked, the explicit determination\footnote{Unlike the situation over a p-adic field, it is not clear what ``explicit determinantion'' means.} of the zeta function for a general homogeneous polynomial $f$ seems to be out of reach, however there are some special polynomials whose local zeta functions can be described in terms gamma factors, namely Sato's \emph{relative invariants} of prehomogeneous vector spaces \cite{Kimurabook} . This is possible because their Bernstein-Sato polynomials $b_f(s)$ are all known. We need the simplest relative invariant.
\begin{proposition} (For the proof, see \cite{Igusa2000} .)
\begin{align*}
\int_{M_n(\mathbb{C})}e^{-\sum_{i,j}|z_{ij}|^2}|{\tt{det}}(z_{ij})|^{2s}\frac{dV}{\pi^{n^2}}=\frac{1}{(2\pi)^{ns}} \prod_{k=1}^{n}\frac{\Gamma(2s+k)}{\Gamma(k)} \ .
\end{align*}
\end{proposition}

Now we may compute the local zeta function of the maximal minors  
$${\tt{det}}_{n+1}\in \mathbb{C}_{n+1}[M_{n+1,N+1}] \ , \ {\tt{det}}_{n}\in \mathbb{C}_{n}[M_{n,N+1}]$$ 

By definition we have that
\begin{align*}
Z({\tt{det}}_{n+1} ; s)&=\frac{\Gamma((n+1)(N+1))}{\Gamma((n+1)(N+1)+(n+1)s)}\int_{M(n+1,N+1)}e^{-\sum_{i,j}|z_{ij}|^2}\frac{|{\tt{det}}_{n+1}(z_{ij})|^{2s}}{\pi^{(n+1)(N+1)}} dV \\
\ \\
&=\frac{\Gamma((n+1)(N+1))}{\Gamma((n+1)(N+1)+(n+1)s)}\int_{M(n+1,n+1)}e^{-\sum_{i,j}|z_{ij}|^2}\frac{|{\tt{det}}_{n+1}(z_{ij})|^{2s}}{\pi^{(n+1)^2}}  {dV} \\
\ \\
&=\frac{\Gamma((n+1)(N+1))}{\Gamma((n+1)(N+1)+(n+1)s)}\frac{1}{(2\pi)^{s(n+1)}}\prod^{n+1}_{k=1}\frac{\Gamma(2s+k)}{\Gamma(k)}\ .
\end{align*}
Therefore we can determine the zeta functions explicitly
\begin{align*}
&\Gamma(a+bs)(2\pi)^{s(n+1)}Z({\tt{det}}_{n+1}\ ;\ s)=\Gamma(a)\prod^{n+1}_{k=1}\frac{\Gamma(2s+k)}{\Gamma(k)} \\
\ \\
&\Gamma(\alpha+\beta s)(2\pi)^{sn}Z({\tt{det}}_{n}\ ;\ s)=\Gamma(\alpha)\prod^{n}_{k=1}\frac{\Gamma(2s+k)}{\Gamma(k)} \\
\ \\
& a=(n+1)(N+1) \ ,\ \alpha=n(N+1) \ ,\  b=(n+1) \ ,\ \beta=n\ .
\end{align*}
Differentiating these identities we have
\begin{align*}
b \Gamma^{\prime}(a)+\Gamma(a)(n+1)\log(2\pi)+\Gamma(a) Z^{\prime}({\tt{det}}_{n+1}\ ; \ 0)=2\Gamma(a)\sum^{n+1}_{k=1}
\frac{\Gamma^{\prime}(k)}{\Gamma(k)}\ .
\end{align*}
 Therefore we see that
 \begin{align*}
 & Z^{\prime}({\tt{det}}_{n+1}\ ; \ 0)=-b\frac{\Gamma^{\prime}(a)}{\Gamma(a)}
 -(n+1)\log(2\pi)+ {2} \sum^{n+1}_{k=1}\frac{\Gamma^{\prime}(k)}{\Gamma(k)} \\
 \ \\
 & Z^{\prime}({\tt{det}}_{n}\ ; \ 0)=-\beta\frac{\Gamma^{\prime}(\alpha)}{\Gamma(\alpha)}
 -n\log(2\pi)+{2} \sum^{n}_{k=1}\frac{\Gamma^{\prime}(k)}{\Gamma(k)}\ .
 \end{align*} 
 Therefore
 \begin{align*}
 Z^{\prime}({\tt{det}}_{n+1} \ ; \ 0)&=-(n+1)\left(-\gamma+\sum^{(n+1)(N+1)-1}_{k=1}\frac{1}{k}\right)+O(1)\\
\ \\
&\sim -(n+1)\log(d)+O(1) \quad \mbox{as}\ \ d\ra\infty \ .
\end{align*}
Similarly we see that
\begin{align*}
Z^{\prime}({\tt{det}}_{n}\ ; \ 0)\sim -n\log(d) +O(1)\quad \mbox{as}\ \ d\ra\infty \ .
\end{align*}
Next we compute the $L^2$ norms.
\begin{align*}
&\log\Gamma(a+\deg(R))+\deg(R)\log(2\pi)+\log Z({\tt{det}}_{n+1} \ ; \ d)=\\ 
\ \\
&\log\Gamma(a)+(n+1) \log\Gamma(2d)+o(1)\ . \\
\ \\
&\log\Gamma(\alpha+\deg(\Delta))+\deg(\Delta)\log(2\pi)+\log Z({\tt{det}}_{n} \ ; \ \frac{\deg(\Delta)}{n})=\\
 \ \\
&\log\Gamma(\alpha)+n\log\Gamma(2\frac{\deg(\Delta)}{n})+o(1)\ .
\end{align*}
Asymptotics of the Gamma function give at once that
\begin{align*}
& \log Z({\tt{det}}_{n+1} \ ; \ d)=\deg(R)\log(d) +O(d) \quad \mbox{as} \ d\ra \infty \\
\ \\
&\log Z({\tt{det}}_{n} \ ; \frac{\deg(\Delta)}{n})=\deg(\Delta)\log(d) +O(d) \quad \mbox{as} \ d\ra \infty\ .
\end{align*}
Now we may address the issue of bounding the height discrepancy $\delta(\sigma)$ of a given $X\subset\cpn$ where $N\sim \mbox{deg}(X)$ are large. Let $\sigma(t)=\lambda(t)$ be a generic algebraic one parameter subgroup of $\slnc$.  ``Generic'' means that $\lambda$ lies in some $T$ satisfying
\ \\
\begin{enumerate}
\item $\mathcal{N}(R_X)=\mathcal{N}(\elam)$ \\
\ \\
\item $\mathcal{N}(\Delta_X)=\mathcal{N}(\emu)$ \\
\ \\
\item $\lambda$ lies in the interior of one of the maximal cones $\tau \subset N_{\mathbb{R}}(T)$ in the common refinement of the normal fans of the weight polyhedra of the representations $\elam$ and $\emu$\ .
\end{enumerate}

In this case we have
\begin{align*}
t^{-w_{\lambda}(R_X)}\lambda(t)\cdot R_{X}={\tt{det}(a_{ij})}^{\frac{\deg(R_X)}{n+1}}+o(t) \quad t\ra 0 \\
\ \\
t^{-w_{\lambda}(\Delta_X)}\lambda(t)\cdot \Delta_{X}={\tt{det}(b_{ij})}^{\frac{\deg(\Delta_X)}{n}}+o(t) \quad t\ra 0 \\
\end{align*}
for some maximal minors ${\tt{det}(a_{ij})}$ and ${\tt{det}(b_{ij})}$ .

By continuity (see \cite{tian97}), the specializations of the heights of $R_X$ and $\Delta_X$ are therefore given by
\begin{align*}
\lim_{t\ra 0} h_{F}(\sigma(t)X)&=\frac{\deg(R)}{n+1}Z^{\prime}({\tt{det}}_{n+1} \ ; \ 0)-\log Z({\tt{det}}_{n+1} \ ; \ d) \\
\ \\
&=-2\deg(R)\log(d)+O(d) \ . \\
\ \\
\lim_{t\ra0} h_{\Delta}(\sigma(t)X)&=\frac{\deg(\Delta)}{n}Z^{\prime}({\tt{det}}_{n} \ ; \ 0)-\log Z({\tt{det}}_{n} \ ; \ \frac{\deg(\Delta)}{n}) \\
\ \\
=&-2\deg(\Delta)\log(d)+O(d) \ .
\end{align*}

This proves the main result of this section, namely that the height discrepancy is $O(d^2)$ along all generic degenerations.
\begin{proposition}\label{boundeddiscrepancy}
For generic $\sigma(t)\in SL(N+1, \mathbb{C}(t))$ we have that
\begin{align*}
\delta(\sigma(0))=\lim_{t\ra 0}|\deg(R)h_{\Delta}(\sigma(t)X)-\deg(\Delta)h_{F}(\sigma(t)X)|=O(d^2)\ .
\end{align*}
\end{proposition}
\begin{remark}
In the above proposition, despite the notation, we do not degenerate $X$ but the associated polynomials. The presence of $\Delta_X$ prevents the compatibility of these two types of degenerations.
\end{remark}

Proposition \ref{boundeddiscrepancy} is significant for the following reason. The left hand bound  
\begin{align*}
-d\left(\sum^{N}_{j=1}\frac{1}{j}\right)\leq h(P) \leq 0
\end{align*}
is assumed by the $d$th power of any linear form. In an attempt to make $|\delta(\sigma)|$ large we push one of the polynomials in the direction of a power of such a form, in our case the determinant of a maximal minor. But the other polynomial \emph{also moves towards a maximal minor}, therefore they {both} become large and their dominant terms match precisely and cancel. This is (one reason) why the author believes that $(\dagger\dagger)$ is unobstructed.

\section{ Asymptotic Stability and Properness of the Mabuchi Energy}
 We quickly collect some definitions surrounding Mabuchi's K-energy map.
Let 
\begin{align*}  
(X^n,\om)  \ , \ n={\tt{dim}}_{\mathbb{C}}(X) \ 
\end{align*}
be a compact K\"ahler manifold. Recall that the K\"ahler form $\om$ is given locally by a Hermitian positive definite matrix of functions
  \begin{align*}
  \om=\frac{\sqrt{-1}}{2\pi}\sum_{ i,j }g_{i\overline{j}}dz_{i}\wedge d\overline{z}_j  \ .
  \end{align*}
The Ricci form of $\om$ is the smooth $(1,1)$ form on $X$ given by
\begin{align*}
\begin{split}
 {\tt{Ric}}(\om):= \frac{-\sqrt{-1}}{2\pi }\dl\dlb\log\det(g_{i\overline{j}}) =\sum_{i,j}\frac{-\sqrt{-1}}{2\pi }R_{i\overline{j}}dz_{i}\wedge d\overline{z}_j \in \Gamma(\Lambda^{1,1}_X)\ .
\end{split}
\end{align*}
The scalar curvature is by definition the contraction of the Ricci curvature 
 \begin{align*}
{\tt{Scal}}(\om): =\sum_{i,j}g^{i\overline{j}}R_{i\overline{j}}\in C^{\infty}(X)\ .
 \end{align*}  
The volume $V$ and the average of the scalar curvature $\mu$ depend only on $[\om]$ and are given by
\begin{align*}
 V=\int_X\om^n \ , \ \mu=\frac{1}{V}\int_X{\tt{Scal}}(\om)\om^n   \ .
 \end{align*}
The space of K\"ahler metrics in the class $[\om]$ is defined by
\begin{align*}
\begin{split}
& \mathcal{H}_{\om}:=\{\vp\in C^{\infty}(X)\ |\ \om_{\vp}:=\om+\frac{\sqrt{-1}}{2\pi}\dl\dlb\vp  >0 \}  \ .
\end{split}
\end{align*} 
  \begin{definition} (Mabuchi \cite{mabuchi}) 
 \emph{The \textbf{\emph{K-energy map}} $\nu_{\om}:\mathcal{H}_{\om}\ra \mathbb{R}$ is given by
\begin{align*}
\nu_{\om}(\varphi):=-\frac{1}{V}\int_0^1\int_{X}\dot{\varphi_{t}}({\tt{Scal}}(\om_{\varphi_{t}})-\mu)\omega_{t}^{n}dt 
\end{align*}
 $\vp_{t}$ is a $C^1$ path in $\mathcal{H}_{\om}$ satisfying $\vp_0=0$ , $\vp_1=\vp$ .}
 \end{definition}
Mabuchi shows that  $\nu_{\om}$ is independent of the path chosen. It is clear that 
 {$\vp$ is a critical point for $\nu_{\om}$ if and only if $$ {\tt{Scal}}(\om_{\varphi})\equiv \mu\ . $$ }
 What is relevant for the present article is the following theorem, first established by Bando and Mabuchi in the case $\mathbb{L}=-K_{X}$ , and then generalized some years later by Donaldson and Li .   
 
 \begin{theorem} (see \cite{bando-mabuchi87}, \cite{donaldson2001}, \cite{donaldson2005}, \cite{chili2011} ) Let $(X,\mathbb{L})$ be a polarized manifold, and assume that there is a constant scalar curvature metric in the class ${\tt{C}_1}(\mathbb{L} )$. Then the Mabuchi energy is bounded below on $\mathcal{H}_{\om}$ where $h$ is any Hermitian metric on $\mathbb{L}$ with positive curvature $\om$.
\end{theorem}
 
We recall the Aubin $J_{\om}$ functional (see \cite{aubin84}) and the associated energy $F_{\om}^o$  
\begin{align*}
& J_{\omega}(\varphi):= \frac{1}{V}\int_{X}\sum_{i=0}^{n-1}\frac{\sqrt{-1}}{2\pi}\frac{i+1}{n+1}\dl\varphi \wedge \dlb
\varphi\wedge \omega^{i}\wedge {\omega_{\varphi} }^{n-i-1}\\
\ \\
& F_{\om}^o(\varphi):=J_{\omega}(\varphi)-\frac{1}{V}\int_{X}\varphi\ \om^n \ .
\end{align*}

\begin{definition}\label{proper}(Tian \cite{tian97})
\emph{Let $(X,\om) $ be a K\"ahler manifold . The Mabuchi energy is \emph{\textbf{proper}} provided there exists constants $\varepsilon >0$ and $b$ such that for all $\varphi\in \mathcal{H}_\om$ we have }
\begin{align*}
\begin{split}
&\nu_{\om}(\vp)\geq \varepsilon  J_{\omega}(\varphi)+b  \ .
 \end{split}
\end{align*}
\end{definition}
\begin{remark}
The constant $\varepsilon$ in this definition is related to the ``stability exponent'' $m$ by $\varepsilon m=1$.
\end{remark}

Let $(X^n,\mathbb{L})$ be a polarized manifold. Let $h$ be a smooth Hermitian metric on $\mathbb{L}$ with positive curvature $\om$. Choose $k$ large enough so that there is an embedding
\begin{align*}
\iota_k:X\ra \mathbb{P}(H^0(X , \mathbb{L}^k)^*) \ .
\end{align*}
We will always assume that the embedding is given by a \emph{unitary} basis of sections $\{S_i \}$. Similarly we outfit $H^0(X , \mathbb{L}^k)$ with the usual $L^2$ Hermitian metric. We let $\om_{FS}$ denote the corresponding Fubini-Study K\"ahler metric on the (dual) projective space of sections. Then
\begin{align*}
\iota^*_{k}\om_{FS}|_{\iota_k(X)}=k\om_h+\frac{\sqrt{-1}}{2\pi}\dl\dlb\log\left(\sum^{N_k}_{i=0}|S_i|^2 \right)
\end{align*}
Let $G=SL(H^0(X , \mathbb{L}^k))$ , then $\sigma\in G$ acts on the sections by
\begin{align*}
\sigma\cdot S_{i}=\sum_{0\leq j\leq N_k}\sigma_{ij}S_j \ .
\end{align*}
Define 
\begin{align*}
\Psi_{\sigma}:=\log \sum_{0\leq i\leq N_k+1}|\sigma\cdot S_i|^2\ .
\end{align*}
The \emph{Bergman metrics} of level $k$ are given by
\begin{align*}
\mathscr{B}_k:=\{ \frac{1}{k}\Psi_{\sigma} \ |\ \sigma\in SL(N_k+1,\mathbb{C}) \} \subset \mathcal{H}_{\om}\ .
\end{align*}
 A key ingredient in this paper is the following result of Tian \cite{tianberg} .
 \begin{theorem} (Tian's Thesis)
 The spaces $\mathscr{B}_k$ are dense in the $\tt{C}^2$ topology
 \begin{align*}
 \overline{\bigcup _{k}\mathscr{B}_k}=\mathcal{H}_{\om} \ .
 \end{align*}
\end{theorem}

Now we are prepared to establish that the asymptotic stability of $(X,\mathbb{L})$ is equivalent to the existence of a cscK metric $\om\in{\tt{C}}_1(\mathbb{L})$ provided that $(\dagger\dagger)$ holds.

We begin with the following lemma, which was shown to the author by Gang Tian.
\begin{lemma}\label{tian'slemma}
There is a uniform constant $C$ such that for all sufficiently large $k\in \mathbb{N}$ we have
\begin{align*}
C+\frac{1}{k}\log\left(\frac{||\sigma||^2}{N_k+1} \right)\leq \int_X\frac{\Psi_{\sigma}}{k}\frac{\om^n}{V_o} \quad  .
\end{align*}
\end{lemma}
\begin{proof}
If $||\sigma||^2:={\tt{Trace}}(\sigma\sigma^*)$  then we observe that the unitarity of the basis gives
\begin{align*}
 \sum_{0\leq i\leq N_k}\frac{||\sigma\cdot S_{i}||^2}{||\sigma||^2}=1 \ .
\end{align*}
Therefore there is an index $j$ such that
\begin{align*}
||\sigma\cdot S_{j}||^2\geq \frac{||\sigma||^2}{N_k+1} \ .
\end{align*}
Define $$T_j^{\sigma}:=\frac{\sigma\cdot S_{j}}{||\sigma\cdot S_{j}||} \ .$$ Let $\alpha(\mathbb{L})$ be Tian's alpha invariant, and choose any $0<\beta < 
\alpha(\mathbb{L})$ then there exists a uniform constant $C(\beta)>0$ such that
\begin{align*}
\int_X\left(\frac{1}{|T_j^{\sigma}|^2}\right)^{\frac{\beta}{k}}\frac{\om^n}{V} \leq C(\beta) \ .
\end{align*}
Therefore we have that
\begin{align*}
-\frac{\beta}{k}\int_{X}\left(\log|\sigma\cdot S_{j}|^2-\log||\sigma\cdot S_{j}||^2\right)\frac{\om^n}{V} \leq \log C(\beta) \ .
\end{align*}
Therefore
\begin{align*}
\frac{1}{k}\log\left(\frac{||\sigma||^2}{N_k+1} \right)-\frac{1}{\beta}\log  C(\beta)\leq \frac{1}{k}\int_X\log|\sigma\cdot S_j|^2\frac{\om^n}{V}
\leq \frac{1}{k}\int_X\log \sum_{0\leq i\leq N_k+1}|\sigma\cdot S_i|^2\frac{\om^n}{V} \ .
\end{align*}
\end{proof}

 We need to compare the Mabuchi energy and the Aubin energy of the reference metric $\om$ and the restrictions of the Fubini-Study metrics coming from the large projective embeddings. It is easy to see, but absolutely crucial for our argument, that $\nu_{\om}$ does not scale but $F_{\om}^o$ does scale as we pass between $\om$ and ${\om_{FS}|_{\iota_k(X)}}$. We collect the precise comparisons below, where $o(1)$ denotes any quantity that converges to 0 as $k\ra \infty$. In fact all of the $o(1)$'s below have the form $O(\frac{\log(k)}{k})$. 
  
 \begin{align*}
 & \nu_{\om}(\frac{\Psi_{\sigma}}{k})=\nu_{\om_{FS}|_{\iota_k(X)}}(\varphi_{\sigma})+o(1) \ .\\
 \ \\
 &J_{\om}(\frac{\Psi_{\sigma}}{k})=\frac{1}{k}J_{\om_{FS}|_{\iota_k(X)}}(\varphi_{\sigma})+o(1) \ .\\
 \ \\
 & \int_{X}\frac{\Psi_{\sigma}}{k}\frac{\om^n}{V_o}=\frac{1}{V}\int_{\iota_k(X)}\frac{\varphi_{\sigma}}{k}\om^n_{FS}+o(1) \ .
 \end{align*}

Proposition (\ref{heart}) shows that asymptotic stability of $(X,\mathbb{L})$ is equivalent to
\begin{align}\label{comparison}
km\left(\log\frac{||\sigma\cdot\Delta||^2}{||\Delta||^2}-\log\frac{||\sigma\cdot R||^2}{||R||^2}\right)\geq q\log\frac{||\sigma||^2}{N_k+1}-\log\frac{||\sigma R||^2}{||R||^2}-Ck^{2n+1}  
\end{align}
where $k>>0$ , $m$ is a fixed positive integer, and $q=\deg(\Delta_X)\deg(R_X)$. Keep in mind that the norms $||\cdot||$ denote the standard $L^2$ norms (induced by $h$) on the representations and $||\sigma||^2=\tt{Trace}(\sigma\sigma^*)$ .  

The basic comparison theorem from the introduction (Theorem A from \cite{paul2012}) gives
\begin{align*}
d^2(n+1)\nu_{\om_{FS}}(\vps)+\delta(\sigma)=\log\frac{||\sigma\cdot\Delta||^2}{||\Delta||^2}-\log\frac{||\sigma\cdot R||^2}{||R||^2} \ .
\end{align*}
Therefore the left hand side of (\ref{comparison}) becomes
\begin{align*}
& km\left(d^2(n+1)\nu_{\om}(\frac{\Psi_{\sigma}}{k}) +d^2(n+1)o(1)+\delta(\sigma)\right) \\
\ \\
&=kmd^2(n+1)\left(\nu_{\om}(\frac{\Psi_{\sigma}}{k}) + o(1)+\frac{\delta(\sigma)}{d^2(n+1)}\right)\\
\ \\
&\sim kmd^2(n+1)\left(\nu_{\om}(\frac{\Psi_{\sigma}}{k}) + o(1)+C\right) \ .
\end{align*}
Above we used our assumption that $(X,\mathbb{L})$ satisfies condition $(\dagger \dagger)$:
\begin{align*}
\frac{|\delta(\sigma)|}{ d^2}=O(1) \ .
\end{align*} 
 
On the right hand side of $(*)$ we have, by Tian's lemma and known results \footnote{ See \cite{zhang} or \cite{gacms} . } concerning the relationship between $R_X$ and $F^o_{\om_{FS}|_{\iota_k(X)}}$ , the following expression
\begin{align*}
&q\left(\log\frac{||\sigma||^2}{N_k+1}-\frac{1}{\deg(R)}\log\frac{||\sigma R||^2}{||R||^2}\right)-Ck^{2n+1}\\
\ \\
&=q\left( \log\frac{||\sigma||^2}{N_k+1}+F_{\om_{FS}|_{\iota_k(X)}}^{o}(\vps)+\frac{(h_{F}(\sigma X)-h_{F}(X))}{\deg(R)}\right)-Ck^{2n+1} \\
\ \\
& \sim q\left( -\log(N_k+1)+J_{\om_{FS}|_{\iota_k(X)}}(\vps) +\frac{(h_{F}(\sigma X)-h_{F}(X))}{\deg(R)} \right)-Ck^{2n+1} \\
\ \\
&= q\left( -\log(N_k+1)+kJ_{\om}(\frac{\Psi_{\sigma}}{k} ) +ko(1)+\frac{(h_{F}(\sigma X)-h_{F}(X))}{\deg(R)}\right)-Ck^{2n+1}\\
\end{align*}

Now divide both sides by $k^{2n+1}$ to get
\begin{align*}
\nu_{\om}(\frac{\Psi_{\sigma}}{k})+o(1) \geq C_1 J_{\om}(\frac{\Psi_{\sigma}}{k} )-C_1\frac{\log(N_k+1)}{k}+\frac{C_1(h_{F}(\sigma X)-h_{F}(X))}{k\deg(R)} -C_2 \ .
\end{align*} 
We remark that $C_{i}=C_i(n,\om,m)$ .

Recall that the heights satisfy at worst a $d\log(d)$ bound, where $d=k^n$. This gives 
\begin{align*}
\frac{h_{F}(\sigma X)-h_{F}(X)}{k\deg(R)}=O\left(\frac{\log(k)}{k}\right) \ .
\end{align*}
Now we may choose any $\varphi\in\mathcal{H}_{\om}$ and any sequence $$\mathscr{B}_k\ni\frac{\Psi_{\sigma}}{k}\xrightarrow{C^2}\varphi \quad \mbox{as} \ k\ra \infty$$ to see that the Mabuchi energy is proper.

Running this backwards shows that properness of the Mabuchi energy implies asymptotic stability of the polarized manifold $(X,\mathbb{L})$.The
 equivalence between asymptotic semistability and a lower bound on the Mabuchi energy is proved in much the same way. 
 
\bibliographystyle{plain} 
\bibliography{ref.bib}

\begin{thebibliography}{10}

\bibitem{atiyah(resolution)}
M.~F. Atiyah.
\newblock Resolution of singularities and division of distributions.
\newblock {\em Comm. Pure Appl. Math.}, 23:145--150, 1970.

\bibitem{aubin84}
T.~Aubin.
\newblock R\'eduction du cas positif...
\newblock {\em Journ. Funct. Anal.}, 57:143--153, 1984.

\bibitem{bando-mabuchi87}
Shigetoshi Bando and Toshiki Mabuchi.
\newblock Uniqueness of {E}instein {K}\"ahler metrics modulo connected group
  actions.
\newblock In {\em Algebraic geometry, Sendai, 1985}, volume~10 of {\em Adv.
  Stud. Pure Math.}, pages 11--40. North-Holland, Amsterdam, 1987.

\bibitem{bernstein}
I.~N. Bern\v{s}te\u{\i}n.
\newblock The possibility of analytic continuation of {$f_{+}{}^{\lambda }$}\
  for certain polynomials {$f$}.
\newblock {\em Funkcional. Anal. i Prilo\v{z}en}, 2(1):92--93, 1968.

\bibitem{bernsteingelfand}
I.N. Bern\v{s}te\u{\i}n and S.I. Gelfan'd.
\newblock Meromorphy of the functions ${P}^{\lambda}$.
\newblock {\em Functional anlysis and Applications}, pages 84--85, 1969.

\bibitem{bostgilletsoule}
J.-B. Bost, H.~Gillet, and C.~Soul\'{e}.
\newblock Heights of projective varieties and positive {G}reen forms.
\newblock {\em J. Amer. Math. Soc.}, 7(4):903--1027, 1994.

\bibitem{cassaigne-maillot}
J.~Cassaigne and V~Maillot.
\newblock Hauter des hypersurfaces et fonctions zeta d'{I}gusa.
\newblock {\em Journal of Number Thoery}, 83:226--255, 2000.

\bibitem{chencheng1}
X.X. Chen and J.~Cheng.
\newblock On the constant scalar curvature k\"ahler metrics, apriori estimates.
\newblock {\em arXiv:1712.06697}, 2017.

\bibitem{chencheng2}
X.X. Chen and J.~Cheng.
\newblock On the constant scalar curvature k\"ahler metrics ii existence
  results.
\newblock {\em arXiv:1801.00656}, 2018.

\bibitem{chencheng3}
X.X. Chen and J.~Cheng.
\newblock On the constant scalar curvature k\"ahler metrics iii-general
  automorphism group.
\newblock {\em arXiv:1801.05907}, 2018.

\bibitem{donaldson2001}
S.K. Donaldson.
\newblock Scalar curvature and projective embeddings, i.
\newblock {\em JDG}, 59:479--522, 2001.

\bibitem{donaldson2005}
S.K. Donaldson.
\newblock Scalar curvature and projective embeddings, ii.
\newblock {\em Quart. J. Math}, 56:345--356, 2005.

\bibitem{gkz}
I.~M. Gelfand, M.~M. Kapranov, and A.~V. Zelevinsky.
\newblock {\em Discriminants, resultants, and multidimensional determinants}.
\newblock Mathematics: Theory \& Applications. Birkh\"auser Boston Inc.,
  Boston, MA, 1994.

\bibitem{Igusa2000}
J.~Igusa.
\newblock {\em An introduction to the theory of Local Zeta Functions},
  volume~14 of {\em Studies in Advanced Mathematics}.
\newblock AMS/IP, 2000.

\bibitem{Kimurabook}
T.~Kimura.
\newblock {\em Introduction to Prehomogeneous Vector Spaces}, volume 215 of
  {\em Translations of Mathematical Monographs}.
\newblock AMS, 2003.

\bibitem{chili2011}
Chi Li.
\newblock Constant scalar curvature k\"ahler metric obtains the minimum of
  k-energy.
\newblock {\em IMRN}, 9:2161 -- 2175, 2011.

\bibitem{mabuchi}
Toshiki Mabuchi.
\newblock {$K$}-energy maps integrating {F}utaki invariants.
\newblock {\em Tohoku Math. J. (2)}, 38(4):575--593, 1986.

\bibitem{git}
D.~Mumford, J.~Fogarty, and F.~Kirwan.
\newblock {\em Geometric invariant theory}, volume~34 of {\em Ergebnisse der
  Mathematik und ihrer Grenzgebiete (2) [Results in Mathematics and Related
  Areas (2)]}.
\newblock Springer-Verlag, Berlin, third edition, 1994.

\bibitem{gacms}
S.~T. Paul.
\newblock Geometric analysis of {C}how {M}umford stability.
\newblock {\em Adv. Math.}, 182(2):333--356, 2004.

\bibitem{ags}
S.~T. Paul and G.~Tian.
\newblock Analysis of geometric stability.
\newblock {\em Int. Math. Res. Not.}, (48):2555--2591, 2004.

\bibitem{paul2012}
Sean~Timothy Paul.
\newblock Hyperdiscriminant polytopes, {C}how polytopes, and {M}abuchi energy
  asymptotics.
\newblock {\em Annals of Math.}, (175), 2012.

\bibitem{paulsergiou2019}
S.T. Paul and K.~Sergiou.
\newblock Fourier-mukai transforms, euler-green currents, and k-stability.
\newblock {\em arXiv:1905.00086}, page~21, 2019.

\bibitem{smirnov2005}
A.~V. Smirnov.
\newblock Projective orbits of reductive groups, and {B}rion polytopes.
\newblock {\em Uspekhi Mat. Nauk}, 60(2(362)):147--148, 2005.

\bibitem{sturmfels2017}
Bernd Sturmfels.
\newblock The hurwitz form of a projective variety.
\newblock {\em Journal of Symbolic Computation}, 79:186--196, 2017.

\bibitem{tianberg}
Gang Tian.
\newblock On a set of polarized {K}\"ahler metrics on algebraic manifolds.
\newblock {\em J. Differential Geom.}, 32(1):99--130, 1990.

\bibitem{tian97}
Gang Tian.
\newblock K\"ahler-{E}instein metrics with positive scalar curvature.
\newblock {\em Invent. Math.}, 130(1):1--37, 1997.

\bibitem{weyman}
Jerzy Weyman.
\newblock {\em Cohomology of vector bundles and syzygies}, volume 149 of {\em
  Cambridge Tracts in Mathematics}.
\newblock Cambridge University Press, Cambridge, 2003.

\bibitem{zhang}
Shouwu Zhang.
\newblock Heights and reductions of semi-stable varieties.
\newblock {\em Compositio Math.}, 104(1):77--105, 1996.

\end{thebibliography}
\end{document}